\documentclass{amsart}
\usepackage{amssymb}
\usepackage{amsthm}
\usepackage{amscd}

\newcommand{\R}{{\mathbb R}}

\newcommand{\Z}{{\mathbb Z}}

\newcommand{\sgn}{{\operatorname{sgn}\,}}
\newcommand{\eps}{\varepsilon}

\newtheorem{prop}{Proposition}
\newtheorem{lem}{Lemma}

\title{Hermite's Identity and the Quadratic Reciprocity Law}
\author{Franz Lemmermeyer}
\address{M\"orikeweg 1, 73489 Jagstzell}
\email{hb3@uni-heidelberg.de}

\begin{document}
\maketitle
In this note we give a proof of the quadratic reciprocity law based on
Gauss's Lemma and Hermite's identity.

Let $p = 2m+1$ and $q = 2n+1$ be odd primes, and let
$A = \{1, 2, \ldots m\}$ and $B = \{1, 2, \ldots n\}$ denote two
half systems modulo $p$ and $q$, respectively.

For each $a \in A$ we have $qa \equiv r_a \bmod p$ for some $0 < r_a < p$,
hence either $r_a \in A$ or $p-r_a \in A$. In particular,
$r_a \equiv \eps_a a' \bmod p$, where $\eps_a = \pm 1$ and $a' \in A$.
Taking the product of these congruences we find
$$ q^{\frac{p-1}2} \cdot a! \equiv \prod \eps_a a' \bmod p, $$
and since $a! = \prod a'$ and $q^{\frac{p-1}2} \equiv (\frac qp) \bmod p$
we obtain
$$ \Big( \frac qp \Big) = \prod_{a \in A} \eps_a. $$

Now $\eps_a = 1$ if $0< r_a < \frac p2$ and $\eps_a = -1$ otherwise; 
on the other hand we see that
$$ \Big\lfloor \frac{2qa}p \Big\rfloor - 2 \Big\lfloor \frac{qa}p \Big\rfloor
   = \begin{cases} 0 & \text{ if } r_a < \frac p2, \\
                   1 & \text{ if } r_a > \frac p2. \end{cases} $$
Thus $\eps_a = (-1)^{\lfloor \frac{2qa}p \rfloor}$, and we have proved

\begin{lem}[Gauss's Lemma]
  $$\Big( \frac qp \Big) =  (-1)^M \quad \text{for} \quad
     M = \sum_{a \in A} \Big\lfloor \frac{2qa}p \Big\rfloor . $$
\end{lem}

Next we transform the sum $M$ modulo $2$. 

\begin{lem}
  We have
  $$ \sum_{a \in A} \Big\lfloor \frac{2qa}p \Big\rfloor
     \equiv \sum_{a \in A} \Big\lfloor \frac{qa}p \Big\rfloor \bmod 2. $$
\end{lem}

\begin{proof}
  The terms $\lfloor \frac{2qa}p \rfloor$ with  $a < \frac p4$
  occur as $\lfloor \frac{q \cdot 2a}p \rfloor$ in the sum on the right.
  We pair the remaining terms  $\lfloor \frac{2qa}p \rfloor$ with
  $a > \frac p4$ with the terms $\lfloor \frac{qa}p \rfloor$ with
  odd values of $a$ in the sum on the right by pairing
  $\lfloor \frac{2qa}p \rfloor$ with $\lfloor \frac{q(p-2a)}p \rfloor$. 
  The claim follows from the observation that the sum of these two
  terms is even; this in turn follows from 
  $\lfloor \frac{2qa}p \rfloor + \lfloor \frac{q(p-2a)}p \rfloor
  = \lfloor \frac{2qa}p \rfloor + \lfloor q - \frac{2qa}p \rfloor
  = \lfloor \frac{2qa}p \rfloor + q - 1 - \lfloor \frac{2qa}p \rfloor
  = q-1$, and we are done.

  Here we have used the fact that
  $\lfloor a-x \rfloor = a-1 = a - 1 - \lfloor x \rfloor$ for all
  natural numbers $a$ and real numbers $x \in \R \setminus \Z$.
  In fact we have  $\lfloor a-x \rfloor = a-1 = a - 1 - \lfloor x \rfloor$
  when $0 < x < 1$, and the claim follows from the fact that
  both sides  have period $1$.  
\end{proof}

Now we know that 
$$  \Big( \frac qp \Big)  = (-1)^\mu \quad  \text{for} \quad 
    \mu = \sum_{a \in A} \Big\lfloor \frac{qa}p \Big\rfloor
    \qquad  \text{and} \qquad 
    \Big( \frac pq \Big)  = (-1)^\nu  \quad  \text{for} \quad 
    \nu = \sum_{b \in B} \Big\lfloor \frac{pb}q \Big\rfloor. $$
This implies
\begin{equation}\label{Emn}
  \Big( \frac pq \Big) \Big( \frac qp \Big) = (-1)^{\mu + \nu}.
\end{equation}
For proving that $\mu + \nu = \frac{p-1}2 \frac{q-1}2 $ (from which
quadratic reciprocity follows) we use Hermite's identity:

\begin{lem}
  For all real values $x \ge 0$ and all natural numbers $n \ge 1$ we have
  \begin{equation}\label{HermI}
    \lfloor x \rfloor + \Big\lfloor x + \frac1n \Big\rfloor
     + \ldots + \Big\lfloor x + \frac{n-1}n \Big\rfloor
     =  \lfloor nx \rfloor.
  \end{equation}  
\end{lem}

Hermite \cite{Hermite} proved this identity using generating functions;
the elementary proof given here can be found in \cite[Ch. 12]{SA}.

\begin{proof}
Consider the function
$$ f(x) = \lfloor x \rfloor + \Big\lfloor x + \frac1n \Big\rfloor
     + \ldots + \Big\lfloor x + \frac{n-1}n \Big\rfloor
     -  \lfloor nx \rfloor. $$
It is immediately seen that $f(x + \frac1n) = f(x)$ and that
$f(x) = 0$ for $0 \le x < \frac1n$. Thus $f(x) = 0$ for all real values
of $x$, and this proves the claim.
\end{proof}

Applying Hermite's identity (\ref{HermI}) with $x = \frac ap$ and
$n = q$ to the sum $\mu$ and using the fact that
$\lfloor \frac ap + \frac bq \rfloor = 0$ whenever $a \in A$ and $b \in B$,
we find 
\begin{align*}
  \mu & = \sum_{a \in A} \Big\lfloor \frac{aq}p \Big\rfloor
    = \sum_{a \in A} \sum_{b=0}^{q-1} \Big\lfloor \frac ap + \frac bq \Big\rfloor 
    = \sum_{a \in A} \sum_{b=n+1}^{q-1}
      \Big\lfloor \frac ap + \frac bq \Big\rfloor \\
  & = \sum_{a \in A} \sum_{b=1}^{n}
      \Big\lfloor \frac ap + \frac{q-b}q \Big\rfloor 
    = \sum_{a \in A} \sum_{b \in B}
     \bigg( \Big\lfloor \frac ap - \frac bq + 1 \Big\rfloor \bigg)
     \qquad \text{and, similarly,} \\
  \nu & = \sum_{b=1}^m \Big\lfloor \frac{bp}q \Big\rfloor
        = \sum_{a \in A} \sum_{b \in B}
        \Big\lfloor \frac bq - \frac ap + 1 \Big\rfloor .
\end{align*}
Clearly $\lfloor \frac ap - \frac bq +1 \rfloor = 1$ if
$\frac ap - \frac bq > 0$ and  $\lfloor \frac ap - \frac bq  +1 \rfloor = 0$
if $\frac ap - \frac bq < 0$; this implies that
$ \lfloor \frac ap - \frac bq + 1 \rfloor
+ \lfloor \frac bq - \frac ap + 1 \rfloor  = 1$, and we find
$$ \mu + \nu =
    \sum_{a \in A} \sum_{b \in B} \Big\lfloor \frac ap - \frac bq + 1 \Big\rfloor
 +  \sum_{a \in A} \sum_{b \in B} \Big\lfloor \frac bq - \frac ap + 1 \Big\rfloor
    = \frac{p-1}2 \frac{q-1}2 . $$

\end{document}